\documentclass[11pt]{article}
\usepackage{amsmath,amsfonts,amssymb,amsthm,pstricks}
\usepackage[small,nohug]{diagrams}

\title{The Derived Category of the Intersection of Four Quadrics}
\author{Nicolas Addington}

\newtheorem*{theorem*}{Theorem}
\newtheorem{lemma}{Lemma}[section]

\numberwithin{equation}{section}

\renewcommand{\AA}{\mathcal A}
\newcommand{\BB}{\mathcal B}
\newcommand{\CC}{\mathbb C}
\newcommand{\DD}{\mathcal D}
\newcommand{\EE}{\mathcal E}
\newcommand{\HH}{\mathcal H}
\newcommand{\II}{\mathcal I}
\newcommand{\JJ}{\mathcal J}
\newcommand{\MM}{\mathcal M}
\newcommand{\Mt}{{\tilde{\mathcal M}}}
\newcommand{\OO}{\mathcal O}
\newcommand{\PP}{\mathbb P}
\renewcommand{\SS}{\mathcal S}
\newcommand{\TT}{\mathcal T}
\newcommand{\ZZ}{\mathbb Z}

\renewcommand{\phi}{\varphi}
\newcommand{\Wedge}{{\textstyle\bigwedge}}

\newcommand{\hook}{\mathbin{\rule[.2ex]{.4em}{.03em}\rule[.2ex]{.03em}{.4em}}}

\newcommand{\ev}{\text{ev}}
\newcommand{\odd}{\text{odd}}
\newcommand{\sm}{\text{sm}}
\newcommand{\sing}{\text{sing}}

\newcommand{\Caldararu}{C\u{a}l\-d\u{a}\-ra\-ru}

\newcommand\Cl{{C\ell}}
\DeclareMathOperator{\Gr}{Gr}
\DeclareMathOperator{\rank}{rank}
\DeclareMathOperator{\Hom}{Hom}
\DeclareMathOperator{\Ext}{Ext}

\DeclareMathOperator{\PGL}{PGL}
\DeclareMathOperator{\End}{End}
\DeclareMathOperator{\corank}{corank}

\begin{document}

\maketitle

\begin{abstract}
The derived category of a general complete intersection of four quadrics in $\PP^{2n-1}$ has a semi-orthogonal decomposition \linebreak $\langle \OO(-2n+9), \dotsc, \OO(-1), \OO, \DD \rangle$, where $\DD$ is the derived category of twisted sheaves on a certain non-algebraic complex 3-fold coming from a moduli problem.  In particular, when $n=4$ we obtain a (twisted) derived equivalence of Calabi-Yau 3-folds predicted by Gross.  This differs from Kuznetsov's result \cite{kuznetsov_quadrics} in that our construction is geometric and avoids non-commutative varieties.
\end{abstract}

\section{Introduction}

We will be interested in the following objects:
\begin{itemize}
\item $V$, a $2n$-dimensional complex vector space, 
\item $\Phi = \PP \mathop{\rm Sym}^2 V^*$, the space of quadrics in $\PP V$,
\item $X$, a general complete intersection of 2, 3, or 4 quadrics in $\PP V$,
\item $L$, the line, plane, or 3-plane that those quadrics span in $\Phi$, and
\item $\MM$, the double cover of $L$ branched over the hypersurface of singular quadrics.
\end{itemize}
We will see that $\MM$ is a moduli space of vector bundles on $X$ and that their derived categories are related.

Let us take a moment to describe $\Phi$.  Over $\CC$, any two quadrics of the same rank are isomorphic.  A quadric of full rank is smooth and $2n-2$-dimensional.  A quadric of corank 1 is a cone from a point to a smooth $2n-3$-dimensional quadric.  A quadric of corank 2 is a cone from a line to a smooth $2n-4$-dimensional quadric, and so on.  $\Phi$ has a stratification $\Phi \supset \Delta_1 \supset \Delta_2 \supset \Delta_3 \supset \dotsb$, where $\Delta_c$ is the space of quadrics of corank at least $c$.  The singular locus of $\Delta_c$ is $\Delta_{c+1}$.  $\Delta_1$ is a hypersurface of degree $2n$, $\Delta_2$ is codimension 3, and $\Delta_3$ is codimension 6.  I imagine the whole situation as looking like this:
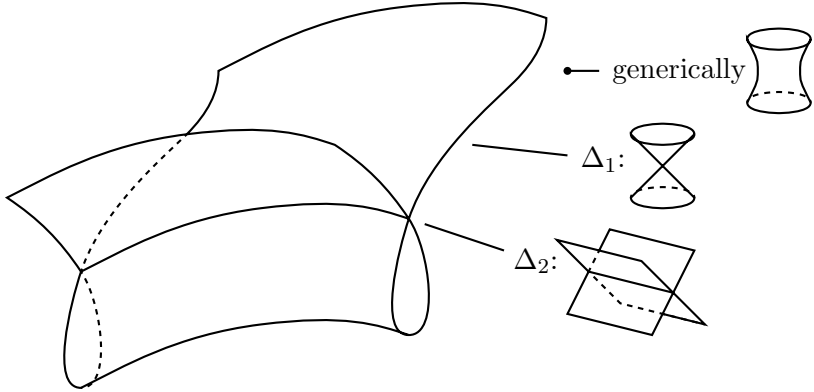
\begin{center}
\psset{unit=4pt,dash=2pt 2pt}
\begin{pspicture}(75,37)
\pscustom{
  \newpath
  \moveto(38,5)
  \rcurveto(-3,1)(-6,2)(-14,1) 
  \rcurveto(-8,-1)(-13,-4)(-17,-6)
  \rcurveto(-3,0)(-1,8)(0,11)
  \rcurveto(-2,3)(-4,5)(-7,7)
  \rcurveto(4,2)(9,5)(17,6)
  \rcurveto(8,1)(11,0)(14,-1)
  \rcurveto(3,-2)(5,-4)(7,-7)
  \rcurveto(2,-3)(3,-11)(0,-11)
  \rcurveto(-3,0)(-1,8)(0,11)
  \rcurveto(2,6)(8,11)(10,13)
  \rcurveto(2,2)(3,4)(3,6)
  \rcurveto(-3,1)(-6,2)(-14,1)
  \rcurveto(-8,-1)(-13,-4)(-17,-6)
  \rcurveto(0,-2)(-1,-4)(-3,-6)
  \moveto(7,11)
  \rcurveto(4,2)(9,5)(17,6)
  \rcurveto(8,1)(11,0)(14,-1) 
}
\pscustom[linestyle=dashed]{
  \newpath
  \moveto(17,24)
  \rcurveto(-2,-2)(-8,-7)(-10,-13)
  \rcurveto(2,-3)(3,-11)(0,-11)
}
\pscircle[fillstyle=solid,fillcolor=black](53,30){1.5pt}
\psline(53,30)(56,30)
\uput[0](56,30){generically}
\pscustom{
  \newpath
  \moveto(70,33)
  \rcurveto(0,1)(3,1)(3,1)
  \rcurveto(0,0)(3,0)(3,-1)
  \rcurveto(0,-1)(-3,-1)(-3,-1)
  \rcurveto(0,0)(-3,0)(-3,1)
  \rcurveto(1,-2)(1,-2)(1,-3)
  \rcurveto(0,-1)(0,-1)(-1,-3)
  \rcurveto(0,-1)(3,-1)(3,-1)
  \rcurveto(0,0)(3,0)(3,1)
  \rcurveto(-1,2)(-1,2)(-1,3)
  \rcurveto(0,1)(0,1)(1,3)
}
\pscustom[linestyle=dashed]{
  \newpath
  \moveto(70,27)
  \rcurveto(0,1)(3,1)(3,1)
  \rcurveto(0,0)(3,0)(3,-1)
}
\psline(44,23)(53,22)
\uput[-10](53,22){$\Delta_1$:}
\pscustom{
  \newpath
  \moveto(59,24)
  \rcurveto(0,1)(3,1)(3,1)
  \rcurveto(0,0)(3,0)(3,-1)
  \rcurveto(0,-1)(-3,-1)(-3,-1)
  \rcurveto(0,0)(-3,0)(-3,1)
  \rlineto(6,-6)
  \rcurveto(0,-1)(-3,-1)(-3,-1)
  \rcurveto(0,0)(-3,0)(-3,1)
  \rlineto(6,6)
}
\pscustom[linestyle=dashed]{
  \newpath
  \moveto(59,18)
  \rcurveto(0,1)(3,1)(3,1)
  \rcurveto(0,0)(3,0)(3,-1)
}
\psline(39.5,15.5)(47,13)
\rput(50,12){$\Delta_2$:}
\pscustom{
  \newpath
  \moveto(62,7)
  \rlineto(4,-1)
  \rlineto(-6,6)
  \rlineto(-8,2)
  \rlineto(3,-3)
  \rlineto(8,-2)
  \rlineto(2,4)
  \rlineto(-8,2)
  \rlineto(-1,-2)
  \moveto(55,11)
  \rlineto(-2,-4)
  \rlineto(8,-2)
  \rlineto(2,4)
}
\pscustom[linestyle=dashed]{
  \newpath
  \moveto(56,13)
  \rlineto(-1,-2)
  \rlineto(3,-3)
  \rlineto(8,-2)
}
\end{pspicture}
\end{center}
but note that $\Delta_2$ is really codimension 3 and $\Delta_3$ is not pictured.

In analogy with Be{\u\i}linson's semi-orthogonal decomposition of the derived category of $\PP V$ \cite{beilinson}
\[ D(\PP V) = \langle \OO_{\PP V}(-2n+1), \dotsc, \OO_{\PP V}(-1), \OO_{\PP V} \rangle, \]
Kapranov \cite{kapranov_square} described the derived category of a smooth, even-dimensional quadric $Q$:
\[ D(Q) = \langle \OO_X(-2n+3), \dotsc, \OO_X(-1), \OO_X, S_+, S_- \rangle. \]	
Here $S_\pm$ are the two ``spinor bundles'' on $Q$, so called because they can be constructed using Cartan's variety of pure spinors.  On $Q^2 \cong \PP^1 \times \PP^1$, they are $\OO(1,0)$ and $\OO(0,1)$, and on $Q^4 \cong \Gr(2,4)$, they are the quotient bundle and the dual of the tautological bundle.  We will discuss them at length in \S2.  Our reference on derived categories, semi-orthogonal decompositions, etc.\ is Huybrechts' book \cite{huybrechts}.

The first to associate a hyperelliptic curve to a pencil of quadrics was Weil \cite{weil}, who computed the zeta function of $X$ from that of $\MM$ to give evidence for his conjectures.  In his thesis \cite{reid}, Reid showed that the variety of $\PP^{n-2}$s on $X$ is isomorphic to the Jacobian of $\MM$ and to the intermediate Jacobian of $X$.  In a series of papers beginning with \cite{desale_ramanan}, Desale and Ramanan described various moduli spaces of bundles on $\MM$ in terms of varieties of linear spaces on $X$.  Bondal and Orlov \cite{bo_semiorth} gave a categorical explanation of this: viewing $\MM$ as the moduli space of spinor bundles on $X$, they used the universal bundle as the kernel of a Fourier-Mukai transform to embed $D(\MM)$ in $D(X)$ and showed that
\[ D(X) = \langle \OO_X(-2n+5), \dotsc, \OO_X(-1), \OO_X, D(\MM) \rangle. \]
Setting $n=2$, so $X$ and $\MM$ are elliptic curves, we recover Mukai's original example \cite{mukai_elliptic} of a derived equivalence.

Tyurin \cite{tyurin} studied intermediate Jacobian of the intersection of three odd-dimensional quadrics and proved a Torelli theorem.  Desale \cite{desale_net} generalized her and Ramanan's work on two quadrics to three.  When $n=3$, so $X$ and $\MM$ are K3 surfaces, Mukai \cite{mukai_K3} showed that $\MM$ is the moduli space of spinor bundles on $X$; this is true for $n>3$ as well.  But now the moduli problem is not fine, and we can only find a twisted pseudo-universal bundle, twisted by some Brauer class $\alpha \in H^2(\MM, \OO_\MM^*)$.  Then:
\[ D(X) = \langle \OO_X(-2n+7), \dotsc, \OO_X(-1), \OO_X, D(\MM, \alpha^{-1}) \rangle. \]
Our reference on twisted sheaves is \Caldararu's thesis \cite{andrei_thesis}, which also discusses the equivalence obtained when $n=3$.

For four or more quadrics, $L$ cannot avoid the corank 2 locus $\Delta_2$, so $L \cap \Delta_1$ must be singular, so $\MM$ must be singular.  In analogy with Ber\u{s}te{\u\i}n--Gelfan'd--Gelfan'd's description of $D(\PP V)$ \cite{bgg}, Kapranov \cite{kapranov_bgg} described of $D(X)$ as a quotient of the derived category of modules over a generalized Clifford algebra, which we will discuss in \S3.  Bondal and Orlov \cite{bo_sheaves} equipped $\MM$ with a related sheaf of algebas $\BB$, viewed it as a non-commutative resolution of singularities, and stated that
\[ D(X) = \langle \OO_X(-2n+2m-1), \dotsc, \OO_X(-1), \OO_X, D(\BB\text{-mod}) \rangle \]
when $n \ge m$, where $m$ is the number of quadrics.  Kuznetsov \cite{kuznetsov_quadrics} proved this and more using his homological projective duality.  For three quadrics, $\BB$ is just a sheaf of Azumaya algebras, so this is equivalent to what we said above about twisted sheaves, but in general $\BB$ is scarier.

We will take a more geometric approach to the intersection of four quadrics, emphasizing the moduli problem.  In \S2 we describe the vector bundles in question, which come from matrix factorizations of the quadrics.  The smooth points of $\MM$ correspond to stable bundles, but $\MM$ also has ordinary double points; each one corresponds to two lines of semi-stable bundles, all S-equivalent.  We can cut out an ODP and replace it by a line in two ways, related by a flop, although in general we cannot hope to do this algebraically.  In \S3 we construct this small resolution of singularities $\Mt \to \MM$ and an $\alpha$-twisted pseudo-universal bundle on $\Mt$.  In \S4 we show that when $n \ge 4$, so $X$ is Fano or Calabi-Yau, this embeds $D(\Mt, \alpha^{-1})$ in $D(X)$.  In \S5 we prove:
\begin{theorem*}
$D(X) = \langle \OO_X(-2n+9), \dotsc, \OO_X(-1), \OO_X, D(\Mt, \alpha^{-1}) \rangle$.
\end{theorem*}
\noindent Setting $n=4$, we obtain a twisted derived equivalence of Calabi-Yau 3-folds predicted by Gross.  For any $n \ge 4$, our construction produces complex manifolds twisted-derived-equivalent to Bondal--Orlov and Kuznetsov's non-commutative varieties.

Such derived equivalences are related to string theory.  Vafa and Witten \cite{vafa_witten} constructed examples of string theory models that appeared to be smooth despite being compactified on singular spaces, and posed the question of how to understand them.  Aspinwall, Morrison, and Gross \cite{amg} observed that all these examples involved a non-trivial Brauer class, which corresponds physically to the so-called $B$-field.  For elliptically-fibered Calabi-Yau 3-folds, \Caldararu\ \cite{andrei_ell_fib} produced a non-K\"ahler small resolution of singularities and obtained derived equivalence twisted by a Brauer class; our result is the analogue in a non--elliptically-fibered case.  In \cite{cdhps} it was conjectured that homologically projectively dual varieties are connected by a gauged linear sigma models, and vice versa.  From this perspective, our example suggests that there should be a GLSM between the smooth, projective $X$ and the non-K\"ahler $\Mt$, taking the twisting into account.

We will end up reproving the results stated above for two and three quadrics.  For five or more quadrics, we will see that $\MM$ has no small resolution, algebraic or otherwise.

I am pleased to thank my advisor, Andrei \Caldararu, who posed this problem to me and taught me the subject, and the many faculty and students at Wisconsin with whom I have had helpful discussions.  This work was supported in part by the National Science Foundation under grants nos.\ DMS-0354112, DMS-0556042, and DMS-0838210.

\section{Spinor Sheaves}

Let us describe in more detail the bundles in our moduli problem.  Let $Q$ be a $2n-2$-dimensional quadric.  On a smooth $Q$ there are two ``spinor bundles'', of rank $2^{n-1}$.  A corank 1 $Q$ is a cone on a smooth $2n-3$-dimensional quadric, on which there is one spinor bundle, again of rank $2^{n-1}$.  We will construct a sheaf on $Q$ that, away from the cone point, is the pullback of this spinor bundle via projection from the cone point.  A corank 2 $Q$ is a cone from a line to a smooth $2n-4$-dimensional quadric, on which there are two spinor bundles $S_\pm$ of rank $2^{n-2}$.  On the smaller quadric there are no extensions between these, but when pulled back to $Q$ and extended across the cone line, there is a $\PP^1$ of non-trivial extensions of $S_+$ by $S_-$ and another of $S_-$ by $S_+$.  Thus we see the double cover of $L$, branched over the singular quadrics, the exceptional lines of $\Mt \to \MM$, and the two possible small resolutions of each ODP, related by a flop.

We will use the ``spinor sheaves'' introduced in \cite{addington_spinor}.  Given a linear space on $Q$, we define a sheaf as follows.

Let $q$ be a quadratic form on $V$, non-zero but possibly degenerate, and $W \subset V$ an isotropic subspace (that is, $q|_W = 0$).  Ultimately we will want $\corank q \le 2$ and $\dim W = n$, but is profitable to develop the theory in general.  Let $\Cl$ be the Clifford algebra, which can be defined either as a quotient of the tensor algebra
\[ \Cl = \TT(V)/\langle v^2 = q(v) \rangle \]
or as a deformation of the exterior algebra
\begin{equation}\label{wedge_hook}
\Cl = \Wedge V \qquad \qquad v\xi = v \wedge \xi + v \hook \xi.
\end{equation}
Let $I$ be the left ideal $I = \Cl \cdot w_1 \dotsm w_m$, where $w_1, \dotsc, w_m$ is a basis for $W$; because $W$ is isotropic, $I$ does not depend on the choice of basis.  Since $\Cl$ is $\ZZ/2$-graded, we can write $I = I_\ev \oplus I_\odd$.  Define a map of vector bundles
\[ \phi: \OO_{\PP V}(-1) \otimes I_\odd \to \OO_{\PP V} \otimes I_\ev \]
by $\phi(v \otimes \xi) = v\xi$.  We remark that if $\psi: \OO_{\PP V}(-1) \otimes I_\ev \to \OO_{\PP V} \otimes I_\odd$ is defined similarly then $\phi$ and $\psi$ constitute a matrix factorization of $q$.  The sheaf we want is $S := \mathop{\rm coker} \phi$.  The following facts are proved in \cite{addington_spinor}:
\begin{itemize}
\item $S$ is supported on $Q$.  Its restriction to $\PP W \cap Q_\sing$ is trivial of rank $2^{\dim V/W}$.  Elsewhere on $Q$ it is locally free of rank $2^{\dim V/W - 1}$.
\item $S$ is reflexive; in fact $S^*(1)$ is a spinor sheaf.
\item If $Q$ is smooth, recall that there is one family of $m$-dimensional isotropic subspaces if $m < n$ and two if $m = n$.  The isomorphism class of $S$ depends only on the family to which $W$ belongs.  If $m = n$, $S$ is one of the two spinor bundles discussed earlier; otherwise it is a direct sum of them.
\item If $Q$ is singular, hence a cone from $Q_\sing$ to a smooth quadric, then the isomorphism class of $S$ depends on $\PP W \cap Q_\sing$ and on the family to which the projection of $\PP W$ to the smooth quadric belongs.  For example, suppose $\dim V = 4$, $\dim W = 2$, $\corank Q = 2$, so we imagine a line on this figure:
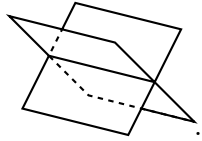
\begin{center}
\psset{unit=5pt}
\begin{pspicture}(14,10)
\pscustom{
  \newpath
  \moveto(10,2)
  \rlineto(4,-1)
  \rlineto(-6,6)
  \rlineto(-8,2)
  \rlineto(3,-3)
  \rlineto(8,-2)
  \rlineto(2,4)
  \rlineto(-8,2)
  \rlineto(-1,-2)
  \moveto(3,6)
  \rlineto(-2,-4)
  \rlineto(8,-2)
  \rlineto(2,4)
}
\pscustom[linestyle=dashed,dash=2pt 2pt]{
  \newpath
  \moveto(4,8)
  \rlineto(-1,-2)
  \rlineto(3,-3)
  \rlineto(8,-2)
}
\end{pspicture}.
\end{center}
If $\PP W$ is not the cone line, it lies on one plane or the other and meets the cone line in a point, and these data determine the isomorphism class of $S$.  That is, the lines on one plane through a fixed point on the cone line all give isomorphic spinor sheaves.
\item If $W$ is maximal, $S$ is stable in the sense of Mumford--Takemoto.
\item If $W$ is not maximal, let $W' \supset W$ be an isotropic space of dimension one greater, and $W''$ an isotropic subspace of the opposite family to $W'$ with $W'' \cap \ker q = W'' \cap \ker q$.  (So in our picture above, $\PP W'$ is the plane containing $\PP W$ and $\PP W''$ the other plane.)  Then there is a short exact sequence
\[ 0 \to S_{W'} \to S_W \to S_{W''} \to 0. \]
Thus the $S$-equivalence class of $S_W$ depends only on $\dim W$.
\item If $H \subset \PP V$ is a hyperplane not containing $\PP W$ then the restriction $S|_{Q \cap H}$ is a spinor sheaf corresponding to the linear space $\PP W \cap H$ on $Q \cap H$.  This justifies what we said earlier about pulling back via projection from $Q_\sing$.
\end{itemize}

All this is proved in \cite{addington_spinor}, but in \S\ref{three_quadrics} we will need the explicit isomorphism between spinor sheaves coming from isotropic subspaces of the same family, so we reproduce it here.
\begin{lemma}\label{ideal_lemma}
Let $W$ and $W'$ be isotropic subspaces of the same dimension, intersecting in even codimension, with $W \cap \ker q = W' \cap \ker q$.  Then $S_W \cong S_W'$.
\end{lemma}
\begin{proof}
Choose a basis $v_1, \dotsc, v_k$ for $W \cap W'$ and bases $w_{k+1},\dotsc, w_m$ and $w'_{k+1}, \dotsc, w'_m$ for the rest of $W$ and $W'$.  Then the maps
\[ I_W \xrightarrow{\cdot w'_{k+1} \dotsm w'_m} I_{W'} \xrightarrow{\cdot w_{k+1} \dotsm w_m} I_W \]
are homomorphisms of graded $\Cl$-modules, and the composition is just multiplication by $\det(q(w_i, w'_j)) \ne 0$.  Since the graded ideals are isomorphic, the spinor sheaves are isomorphic.
\end{proof}

\section{Construction of the Small Resolution and Pseudo-Universal Bundle}

\subsection{Varying the Quadric and Isotropic Subspace}\label{sheaf_of_algs}
Having defined spinor sheaves on one quadric and understood how they depend on a choice of isotropic space, we wish to study what happens in a family of quadrics.

First we vary the Clifford algebra.  Define a sheaf $\AA$ of graded algebras on the space $\Phi$ of quadrics:
\begin{align*}
\AA_0 &= \OO_\Phi \otimes \Wedge^0 V \\
\AA_1 &= \OO_\Phi \otimes \Wedge^1 V \\
\AA_2 &= \OO_\Phi \otimes \Wedge^2 V \ \oplus\ \OO_\Phi(1) \otimes \Wedge^0 V \\
\AA_3 &= \OO_\Phi \otimes \Wedge^3 V \ \oplus\ \OO_\Phi(1) \otimes \Wedge^1 V \\
\AA_4 &= \OO_\Phi \otimes \Wedge^4 V \ \oplus\ \OO_\Phi(1) \otimes \Wedge^2 V \ \oplus\ \OO_\Phi(2) \otimes \Wedge^0 V \\
&\dots                                                      
\end{align*}
The multiplication $\AA_1 \otimes \AA_1 \to \AA_2$ is like \eqref{wedge_hook} above, determined by
\begin{align*}
V \otimes V& \to \Wedge^2 V&		\OO_\Phi(-1) \otimes V &\otimes V \to \OO_\Phi \\
v \otimes v'& \mapsto v \wedge v'&	q \otimes v &\otimes v' \mapsto q(v,v').
\end{align*}
Observe that $\rank \AA_0 = 1$, $\rank \AA_1 = 2n$, and the ranks of the graded pieces grow for a while but eventually stabilize: for $k \ge 2n-1$ we have $\rank \AA_k = 2^{2n-1}$ and $\AA_{k+2} = \AA_k(1)$.  The fiber of $\AA$ over a point $Q \in \Phi$ is not the Clifford algebra but the $\ZZ$-graded Clifford algebra that Kapranov considers in \cite{kapranov_square}:
\[ \AA|_Q = \TT(V)[h]/\langle v^2 = q(v) h \rangle \qquad \qquad \deg h = 2. \]
This algebra is Koszul dual to the coordinate ring of $Q$, at least when $Q$ is smooth.  When $k \ge 2n-1$, its $k^\text{th}$ graded piece is isomorphic to the odd or even piece of the usual Clifford algebra.  More generally, if $X$ is a complete intersection of quadrics $Q_1, \dotsc, Q_m$ and $L \subset \Phi$ is the linear space they span, $\Gamma(\AA|_L)$ is the generalized $\ZZ$-graded Clifford algebra that Kapranov considers in \cite{kapranov_bgg}:
\[ \Gamma(\AA|_L) = \TT(V)[h_1, \dotsc, h_m]/\langle v^2 = q_1(v) h_1 + \dotsb + q_m(v) h_m \rangle. \]
This is Koszul dual to the coordinate ring of $X$.

Next we vary the ideal.  Let $\Phi' = \{ (W,Q) \in \Gr(n,V) \times \Phi : \PP W \subset Q \}$ be the relative Grassmannian of isotropic $\PP^{n-1}$s.  This is smooth, since it is a projective bundle over $\Gr(n,V)$.  Let $p: \Phi' \to \Phi$ be the natural map and $\II \subset p^* A$ be the sheaf of left ideals which, over a point $(W,Q) \in \Phi'$, is generated by the line $\Wedge^n W \subset \Wedge^n V \subset \AA_n|_Q$.  If $k \ge 2n-1$, the fiber of $\II_k$ over a point $(W,Q)$ is $I_\odd$ or $I_\ev$ from the previous section, and $\II_{k+2} = \II_k \otimes p^* \OO_\Phi(1)$.

Last we vary the spinor sheaf: on $\PP V \times \Phi'$, let $\SS$ be the cokernel of the map
\[ \OO_{\PP V}(-1) \boxtimes \II_{2n-1} \to \OO_{\PP V} \boxtimes \II_{2n} \]
given by $v \otimes \xi \mapsto v\xi$.  Its restriction to a slice $\PP V \times (W,Q)$ is the spinor sheaf from the previous section.

\subsection{Restriction to the linear system $L$}
Now if $L \subset \Phi$ is a general linear system, its base locus $X$ avoids the singularities of each $Q \in L$, so the restriction of $\SS$ to $X \times p^{-1}L \subset \PP V \times \Phi'$ is a vector bundle.  We would like it to be the universal bundle for our moduli problem, but many points of $p^{-1}L$ parametrize the same bundle on $X$.

The branched cover $\MM$ of $L$ is the Stein factorization of $p^{-1}L \to L$:
\begin{diagram}
& & p^{-1} L & \rInto & \Phi' \\
& \ldTo \\
\MM & & \dTo & & \dTo_p \\
& \rdTo \\
& & L & \rInto & \Phi.
\end{diagram}
If $L$ is a line or a plane, the points in a fiber of $p^{-1}L \to \MM$ all parametrize the same spinor sheaf, and points of different fibers parametrize different spinor sheaves, so what we want is a section of $p^{-1}L \to \MM$.  When $L$ is a line there is a section; when $L$ is a plane there are only local sections; when $L$ is a 3-plane, these local sections will let us resolve the singularities of $\MM$.

\subsection{Two Quadrics}
If $L$ is a line, Reid \cite{reid} shows that $X$ contains $\PP^{n-2}$s; let $\Pi$ be one of them.  On a smooth $Q \in L$, there are two $\PP^{n-1}$s containing $\Pi$, one from each family.  On a singular $Q$ there is only one, namely the span of $\Pi$ and the cone point.  Thus from $\Pi$ we get a section of $p^{-1}L \to \MM$.  We obtain the desired universal sheaf on $X \times \MM$ by pulling back $\SS$ from $\PP V \times \Phi'$.  Different choices of $\Pi$ give us different universal sheaves on $X \times \MM$, but different universal sheaves can only differ by the pullback of a line bundle from $\MM$, which reflects Reid's result that the space of $\PP^{n-2}$s on $X$ is isomorphic to $\mathop{\rm Pic}^0\MM$.

\subsection{Three Quadrics}\label{three_quadrics}
If $L$ is a plane then $X$ does not generally contain a $\PP^{n-2}$\footnote{The case where $n=3$ and $X$ contains a line is treated by Ingalls and Khalid \cite{ingalls_khalid}.} so we will not be able to produce a global section of $p^{-1}L \to \MM$ as we did above, but we will be able to produce local sections, as follows.  What was important about $\Pi$ above was that it lay on every $Q \in L$ and that it avoided the cone points of the singular $Q$s.  Now we will let $\Pi(Q)$ vary with $Q$.  To be precise,
\begin{lemma}\label{Pi_lemma}
$L$ can be covered by analytic open sets $U$ on which there are maps $\Pi: U \to \Gr(n-1,V)$ with $\Pi(Q) \subset Q_\sm$ for each $Q \in U$.
\end{lemma}
\begin{proof}
Let $Q_0 \in L$.  Then the singular locus of $Q_0$ is at most a point, so we can choose a hyperplane $H \subset \PP V$ such that $Q_0 \cap H$ is smooth.  Let $\Psi$ be the space of quadrics in $H$ and $f: \Phi \rDashto \Psi$ the rational map given by intersecting with $H$, which is regular near $Q_0$.  Let $q: \Psi' \to \Psi$ be the relative Grassmannian of isotropic $\PP^{n-2}$s.
\begin{diagram}
 & & \phantom{'}\Psi' & \rTo & \Gr(n-1,V) \\
 & & \dTo_q \\
Q_0 \in \Phi & \rDashto^{\quad f\quad} & \Psi.
\end{diagram}
$\PGL_{2n-1}\CC$ acts on $\Psi$ and $\Psi'$, and $q$ is equivariant.  $\Psi'$ is smooth as $\Phi'$ was, so we can apply theorems from differential geometry.  By Sard's theorem, $q$ is a submersion almost everywhere.  Since the orbit of $f(Q_0)$ is open in $\Psi$, $q$ is a submersion over $f(Q_0)$.  Thus by the implicit function theorem there is an analytic neighborhood $W$ of $f(Q_0)$ over which $q$ has a local section.

Let $U = L \cap f^{-1}W$ and $\Pi: U \to \Gr(n-1, V)$ be the obvious composition.
\end{proof}

Now on a smooth $Q \in L$ there are two $\PP^{n-1}$s containing $\Pi(Q)$, and on a singular $Q$ there is only one, so we get local sections of $p^{-1}L \to \MM$.  If we pull back $\II_k$, $k \ge 2n-1$, from $\PP V \times \Phi'$ to $U \times \MM$ via these local sections, Lemma \ref{ideal_lemma} tells us how to glue together the resulting local sheaves on a pairwise intersection $U_i \cap U_j$, perhaps after shrinking the $U$s.  In fact Lemma \ref{ideal_lemma} glues $\II_k$ to $\II_{k+2l}$ for some $l > 0$, but these are just twists of each other, hence are locally isomorphic.  But we could not have glued for $k < 2n-1$.

We cannot hope for these gluings to match up on a triple intersection $U_i \cap U_j \cap U_{k'}$, so we do not get an honest sheaf on $X \times \MM$ but only a twisted sheaf, twisted by the pullback of a class $\alpha \in H^2(\MM, \OO_\MM^*)$.  (A priori $\alpha$ appears to take values in $\AA ut(\II_k)$, which is bigger than $\OO_\MM^*$, but because our ideals are simple modules \cite{addington_spinor}, in fact we only need $\OO_\MM^*$.)  We will call this twisted pullback $\II_k$ again and hope that no confusion results.  We can also pull back $\SS$ and use the same gluing, and still have
\[ \OO_\MM(-1) \boxtimes \II_{2n-1} \to \OO_\MM \boxtimes \II_{2n} \to \SS \to 0. \]

\subsection{Four Quadrics}\label{four_quadrics}
If $L$ is a 3-plane then $\MM$ has ODPs, and the points in the fiber of $p^{-1}L \to \MM$ over a singular point of $\MM$ do not all parametrize the same sheaf on $X$, but our construction from the previous subsection will resolve both problems.  The proof of Lemma \ref{Pi_lemma} goes through with a one modification: the singular locus of a $Q \in L$ may now be a line, so $H$ must be codimension 2.  On a smooth $Q \in L$ there are two $\PP^{n-1}$s containing $\Pi(Q)$, on a corank 1 $Q$ there is one, and on a corank 2 $Q$ there is a whole line of them: for each point in the cone line of $Q$, we can take its span with $\Pi(Q)$.

Thus we have we have local {\it rational} sections of $p^{-1}L \to \MM$ that are regular away from the singular points of $\MM$.  Let
\[ \Gamma = \{ (W, Q) \in \Gr(n-1,V) \times U : \Pi(Q) \subset \PP W \subset Q \} \]
be the graph of one, which in the Appendix A we show to be smooth.  As we have said, $\Gamma \to \MM$ contracts some lines over the singular points of $\MM$, which are isolated, and is an isomorphism elsewhere.  Thus we can cut out the ODPs of $\MM$ and glue in $\Gamma$ to produce the resolution of singularities $\Mt \to \MM$.  The two small resolutions of each ODP are again visible here: in our modification of Lemma \ref{Pi_lemma}, the fiber of $\Psi' \to \Psi$ over $f(Q_0)$ now has two connected components.

Now the points of $\Mt$ are in one-to-one correspondence with the isomorphism classes of spinor sheaves on $X$.  As in the previous subsection we get twisted sheaves $\II_k$ on $\Mt$ for $k \ge 2n-1$, and $\SS$ on $X \times \Mt$.

\subsection{Five or More Quadrics}
If $L$ is a 4-plane, the singular locus of $\MM$ is a curve $C$, which we can identify with its image in $L$.  To construct a small resolution as above, for each corank 2 $Q \in C$ we would have to choose (continuously) one of the two families of $\PP^{n-2}$s on $Q_\sm$, or equivalently one of the two families of $\PP^n$s on $Q$.  But this is impossible; the associated double cover of $C$ has no section, as follows.  Consider $\{ (\Lambda, Q) \in \Gr(n+1,V) \times \Phi : \Lambda \subset Q \}$, the relative Grassmannian of isotropic $\PP^n$s.  This is a projective bundle over $\Gr(n+1,V)$, hence is irreducible.  The image of the natural map to $\Phi$ is the locus $\Delta_2$ of quadrics of corank at least 2, since a corank 1 quadric contains only $\PP^{n-1}$s.  Since $\Delta_2$ is codimension 3, the preimage of a general 4-plane $L \subset \Phi$ is irreducible by Bertini's theorem.

Not only does our construction of a small resolution of $\MM$ fail, but $\MM$ has no small resolution whatsoever, as follows.  Let $Q \in C$ be any corank 2 quadric, and choose a general 3-plane $L' \subset L$ through $Q$.  The preimage of this in $\MM$ is a 3-fold with ODPs; the two small resolutions of each are in natural bijection with the two families of $\PP^n$s on $Q$.  Thus a small resolution of $\MM$ would give a continuous choice of a family of $\PP^n$s on each $Q \in C$, which we just saw is impossible.

A small resolution of $\MM$ would when $\dim L \ge 5$ would give one when $\dim L = 5$, so this too is impossible.  Note that when $\dim L \ge 7$ the singular locus of $\MM$ is no longer smooth.

\section{Embedding of \texorpdfstring{$D(\Mt,\alpha^{-1})$}{D(M,alpha\textasciicircum-1)}}
In this section we will show that the Fourier-Mukai transform
\[ F_S: D(\Mt,\alpha^{-1}) \to D(X) \]
is an embedding, where $\SS$ is the $\alpha$-twisted sheaf on $X \times \Mt$ constructed in the \S\ref{four_quadrics}.  It suffices to check that for any $x, y \in \Mt$,
\[ \Hom(F_\SS \OO_x, F_\SS \OO_y[i]) = \begin{cases}
\CC & \text{if } x = y \text{ and } i = 0 \\
0 & \text{if } x \ne y \text{ or } i < 0 \text{ or } i > 3.
\end{cases} \]
This criterion was proved by Bondal and Orlov \cite{bo_semiorth} for untwisted sheaves on algebraic varieties, but it applies equally well to twisted sheaves and to complex manfolds \cite{andrei_thesis}.

If $x \in \Mt$ then $F_\SS \OO_x$ is the restriction to $X$ of a spinor sheaf $S$ on a quadric $Q$ containing $X$.  In particular $F_\SS \OO_x$ is a sheaf, not a complex of sheaves, so $\Hom(F_\SS \OO_x, F_\SS \OO_y[i]) = \Ext^i_X(F_\SS \OO_x, F_\SS \OO_y) = 0$ for $i < 0$.  The singular locus $Q_\sing$ of $Q$ is at most a line and does not meet $X$, and $S$ is a vector bundle except perhaps at one point of $Q_\sing$.

We will use two additional facts from \cite{addington_spinor}:
\begin{itemize}
\item On $\PP V$ there is an exact sequence
\begin{equation}\label{res_on_P}
0 \to \OO_{\PP V}^N(-1) \to \OO_{\PP V}^N \to S \to 0
\end{equation}
where $N = 2^{n-1}$.
\item On $Q$ there is an exact sequence
\begin{equation}\label{res_on_Q}
0 \to S(-2) \to \OO_Q^N(-1) \to \OO_Q^N \to S \to 0.
\end{equation}
\end{itemize}

\subsection{Hom and Ext between spinor sheaves from different quadrics}
First suppose that $S_1$ and $S_2$ are spinor sheaves on different quadrics $Q_1$ and $Q_2$.  Since $S_1^*(1)$ is a spinor sheaf, we have resolutions
\begin{gather*}
0 \to \OO_{\PP V}^N(-2) \to \OO_{\PP V}^N(-1) \to S_1^* \to 0 \\
0 \to \OO_{\PP V}^N(-1) \to \OO_{\PP V}^N \to S_2 \to 0.
\end{gather*}
Choose quadrics $Q_3, Q_4 \in L$ that intersect each other and $Q_1$ and $Q_2$ transversely and avoid the points where $S_1$ and $S_2$ fail to be vector bundles (there are at most two).  Restrict one of the resolutions above to $Q_3 \cap Q_4$ and tensor it with the other to get a resolution
\[ 0 \to \OO_{Q_3 \cap Q_4}^{N^2}(-3) \to \OO_{Q_3 \cap Q_4}^{2N^2}(-2) \to \OO_{Q_3 \cap Q_4}^{N^2}(-1) \to (S_1^* \otimes S_2)|_X \to 0. \]
Now from the Koszul complex
\[ 0 \to \OO_{\PP V}(-4) \to \OO_{\PP V}(-2)^2 \to \OO_{\PP V} \to \OO_{Q_3 \cap Q_4} \to 0 \]
we find that $\OO_{Q_3 \cap Q_4}(-i)$ has no cohomology for $1 \le i \le 3$ (we required $n \ge 4$), so $\Ext^*_X(S_1|_X, S_2|_X) = H^*((S_1^* \otimes S_2)|_X) = 0$.

\subsection{Hom between spinor sheaves from the same quadric}
Next suppose that $S$ and $S'$ are two spinor bundles on the same quadric $Q_1$.  From \eqref{res_on_P} we see that
\begin{equation}\label{coho_vanishes}
0 = H^*(S'(-1)) = H^*(S'(-2)) = \dotsb = H^*(S'(-2n+2)).
\end{equation}
Applying $\Hom_{Q_1}(-,S')$ to twists of \eqref{res_on_Q} we see that
\begin{equation}\label{chain_of_twists}
\Ext^i_{Q_1}(S,S') = \Ext^{i+2}_{Q_1}(S,S'(-2)) = \dotsb = \Ext^{i+2n-2}_{Q_1}(S,S'(-2n+2)).
\end{equation}
In particular,
\begin{multline*}
0 = \Ext^{<2}_{Q_1}(S,S'(-2)) = \Ext^{<4}_{Q_1}(S,S'(-4)) = \Ext^{<6}_{Q_1}(S,S'(-6)).
\end{multline*}
Choose quadrics $Q_2, Q_3, Q_4 \in L$ that meet each other and $Q_1$ transversely and tensor $S'$ with the Koszul complex of $Q_2 \cap Q_3 \cap Q_4$ to get
\[ 0 \to S'(-6) \to S'(-4)^3 \to S'(-2)^3 \to S' \to S'|_X \to 0. \]
Applying $\Hom(S, -)$ and using the facts above, we find that $\Hom(S,S'|_X) = \Hom(S,S')$.  In \cite{addington_spinor} it is shown that $\Hom(S,S) = \CC$ for the spinor sheaves considered here.  The same proof is easily adapted to show that $\Hom(S,S') = 0$ if $S \ne S'$.  Of course $\Hom(S|_X,S'|_X) = \Hom(S,S'|_X)$.

\subsection{Ext between different spinor sheaves from the same quadric}
If $S$ and $S'$ are distinct spinor bundles on $Q_1$ then either they are both vector bundles or they fail to be so at distinct points, so we can let $E = S^* \otimes S'$ and rewrite \eqref{chain_of_twists} as
\[ H^i(E) = H^{i+2}(E(-2)) = \dotsb = H^{i+2n-2}(E(-2n+2)). \]
Above we saw that $H^0(E) = 0$, and since $\dim Q_1 = 2n-2$, we see that $E, E(-2), \dotsc, E(-2n+2)$ have no cohomology.  Thus
\[ 0 \to E(-6) \to E(-4)^3 \to E(-2)^3 \to E \to E|_X \to 0 \]
is a resolution of $E|_X$ by sheaves with no cohomology, so $\Ext^*_X(S|_X,S'|_X) = H^*(E|_X) = 0$.

\subsection{Ext of a spinor sheaf with itself}
If $S$ on $Q_1$ fails to be a vector bundle at some point, choose a $Q_2 \in L$ that avoids that point and intersects $Q_1$ transversely.  If we now let $E = (S^* \otimes S')|_{Q_1 \cap Q_2}$, we only get
\[ H^i(E) = H^{i+2}(E(-2)) = \dotsb = H^{i+2n-4}(E(-2n+4)). \]
Since $\dim Q_1 \cap Q_2 = 2n-3$, we have
\[ H^{>1}(E) = H^{>3}(E(-2)) = H^{>5}(E(-4)) = 0. \]
Thus from
\[ 0 \to E(-4) \to E^2(-2) \to E \to E|_X \to 0 \]
we find that $\Ext^i_X(S|_X, S|_X) = H^i(E|_X) = 0$ for $i > 3$.

\section{Semi-Orthogonal Decomposition of $D(X)$}

Recall that our goal is to prove that
\[ D(X) = \langle \OO_X(-2n+9), \dotsc, \OO_X(-1), \OO_X, D(\Mt, \alpha^{-1}) \rangle. \]
In the Calabi-Yau case $n=4$ there are no line bundles, so in fact we have a derived equivalence.  If we were only interested in this case, we would be done:
\begin{theorem*}[Bridgeland \cite{bridgeland_omegaX}]
A fully faithful Fourier-Mukai transform $F_\SS: D(\Mt, \alpha^{-1}) \to D(X)$ is an equivalence if for any $x \in \Mt$ one has
\[ F_\SS \OO_x \otimes \omega_X \cong F_\SS \OO_x. \]
\end{theorem*}

In the Fano case, from the Koszul complex of $X$ we see that $\OO_X(-2n+9), \dotsc, \OO_X$ is a strong exceptional collection.  From \eqref{res_on_P} we see that they are right orthogonal to the spinor sheaves, which are the Fourier-Mukai transforms of the skyscraper sheaves on $\Mt$, hence to all of $D(\Mt, \alpha^{-1})$.\footnote{In this section we are using facts from \cite{huybrechts} quite freely.}

We will show that these line bundles and $D(\Mt, \alpha^{-1})$ generate the line bundles $\OO_X(-d)$ for all $d \ge 0$, hence generate all of $D(X)$.

Recall that $\II_k$, $k \ge 2n-1$ are twisted sheaves on $\Mt$.  We will take the Fourier-Mukai transform of $\II_{2n}^*(n-4)$.  We will want to have this diagram visible:
\begin{diagram}
& & X \times \Mt \\
 & \ldTo & & \rdTo \\
X & & & & \Mt.
\end{diagram}
On $X \times \MM$ there is an exact sequence
\[ 0 \to \SS \to \OO_X(1) \boxtimes \II_{2n+1} \to \OO_X(2) \boxtimes \II_{2n+2} \to \dotsb. \]
If $k \ge 2n-1$ and $l \ge 0$ then $\II_{k}(l) = \II_{k+2l}$, so when we tensor with the pullback of $\II_{2n}^*(n-4)$, we get
\begin{multline*}
0 \to \SS \otimes \pi_\MM^* \II_{2n}^*(n-4) \to \OO_X(1) \boxtimes \HH om_\Mt(\II_{2n},\II_{4n-7}) \\
\to \OO_X(2) \boxtimes \HH om_\MM(\II_{2n},\II_{4n-6}) \to \dotsb.
\end{multline*}
For brevity, let $A_j = \Gamma(\AA_j|_L)$, which we discussed in \S\ref{sheaf_of_algs}.  It is reasonable to hope that for $j \ge 2n-7$, the natural map $A_j \to \Hom_\Mt(\II_{2n},\II_{2n+j})$ is an isomorphism and $\Ext^i_\Mt(\II_{2n},\II_{2n+j}) = 0$.  In Appendix B we show that this is the case.

Now when we push down to $X$ we have
\[ 0 \to F_\SS(\II_{2n}^*(n-4)) \to \OO_X(1) \otimes A_{2n-7} \to \OO_X(2) \otimes A_{2n-6} \to \dotsb. \]
But the sequence
\[ 0 \to \OO_X \otimes A_0 \to \OO_X(1) \otimes A_1 \to \OO_X(2) \otimes A_2 \to \dotsb \]
is exact, so we have a resolution
\begin{multline*}
0 \to \OO_X(-2n+8) \otimes A_0 \to \OO_X(-2n+9) \otimes A_1 \\
\to \dotsb \to \OO_X \otimes A_{2n-8} \to F_\SS(\II_{2n}^*(n-4)) \to 0
\end{multline*}
so we can generate $\OO_X(-2n+8) = \omega_X$.  Similarly, from $F_\SS(\II_{2n+1}^*(n-4))$ we can generate $\OO_X(-2n+7)$, and similarly all the negative line bundles.

\appendix

\section{Appendix: Smoothness of \texorpdfstring{$\Gamma$}{Gamma}}
In this appendix we give the rather messy proof that the space $\Gamma$ constructed in \S\ref{four_quadrics} is smooth.  Let
\[ \Phi'' = \{ (\Pi, Q) \in \Gr(n-1,V) \times \Phi : \Pi \subset Q \} \]
be the relative Grassmanian of isotropic $\PP^{n-2}$s,
\[ \Phi''' = \{ (\Pi, \Lambda, Q) \in \Gr(n-1,V) \times \Gr(n,V) \times \Phi : \Pi \subset \Lambda \subset Q \} \]
the relative flag variety, and $p': \Phi'' \to \Phi$ and $p'': \Phi''' \to \Phi''$ the natural maps, with apologies for all the primes.  We produced an analytic open set $U \subset \Phi$ and a map $\Pi: U \to \Gr(n-1,V)$ with $\Pi(Q) \subset Q_\sm$ for every $Q \in U$.  Note that $\Pi$ stands for plane, not projection.  We summarize the situation in a diagram:
\begin{diagram}
& & \phantom{'''}\Phi''' &= \{ \makebox[2in][l]{flags $\PP^{n-2} \subset \PP^{n-1} \subset$ quadric$\}$} \\
& & \dTo_{p''} \\
& & \phantom{''}\Phi'' &= \{ \makebox[2in][l]{flags $\PP^{n-2} \subset$ quadric$\}$} \\
& \ruTo^s & \dTo_{p'} \\
U & \rInto & \Phi &= \{ \makebox[2in][l]{quadrics$\}$}
\end{diagram}
where $s(Q) = (\Pi(Q),Q)$.  Then $\Gamma$ was the fiber product $(L \cap U) \times_{\Phi''} \Phi'''$.  Since $\Phi''$ and $\Phi'''$ are smooth, to show that $\Gamma$ is smooth it suffices to show that $s|_L$ and $p''$ are transverse, that is, for each $Q \in L \cap U$ and $\Lambda \in \Gr(n,V)$ with $\Pi(Q) \subset \Lambda \subset Q$, $s_* T_Q L$ and $p''_* T_{(\Pi(Q),\Lambda,Q)} \Phi'''$ span $T_{s(Q)} \Phi'$.

$\PGL_{2n}\CC$ acts on $\Phi$ and all its primes, and that all the projections are equivariant.  Let us describe the orbits:
\begin{itemize}
\item The orbits of $\Phi$ are determined by the corank of $Q$, or equivalently by the dimension of $Q_\sing$.
\item The orbits of $\Phi''$ are determined by the type of the flag
\[ \PP^{n-2} \cap Q_\sing\ \subset\ Q_\sing. \]
For example, the following are orbits: smooth quadrics with any $\PP^{n-2}$; corank 1 quadrics with $\PP^{n-2}$s avoiding the cone point; corank 1 quadrics with $\PP^{n-2}$s containing the cone point; corank 2 quadrics with $\PP^{n-2}$s meeting the cone line in a point; corank 2 quadrics with $\PP^{n-2}$s containing the cone line.
\item The orbits of $\Phi'''$ are determined by the type of the flag
\[ \PP^{n-2} \cap Q_\sing\ \subset\ \PP^{n-1} \cap Q_\sing\ \subset\ Q_\sing. \]
If $O$ is an orbit in $\Phi''$ with $\corank Q \le 2$ then $p''^{-1}(O)$ is a single orbit in $\Phi'''$.
\end{itemize}

Let $O \subset \Phi'''$ be the orbit of $(\Pi(Q),\Lambda,Q)$; then $p''(O)$ and $p'(p''(O))$ are also orbits, so all are smooth manifolds.  By Sard's theorem, $O \to p''(O)$ is a submersion, as is $p''(O) \to p'(p''(O))$.  Since $s$ is a section of $p'$,
\[ T_{s(Q)} \Phi'' = \ker p'_* + s_* T_Q \Phi. \]
Since $L$ is transverse to $p'(p''(O))$,
\[ T_{s(Q)} \Phi'' = \ker p_* + s_* T_Q p'(p''(O)) + s_* T_Q L. \]
Now $s_* T_Q p'(p''(O)) \subseteq T_{s(Q)} s(p'(p''(O)))$, and by our construction of $s$, $s(p'(p''(O))) \subset p''(O)$.  Also, since $p'_*$ is surjective, $\ker p'_* = T_{s(Q)} p'^{-1}(Q)$.  But $p''(O)$ is open in $p'^{-1}(p'(p''(O))$, so
\[ T_{s(Q)} \Phi'' = T_{s(Q)} p''(O) + s_* T_Q L. \]
By Sard's theorem, $O \to p''(O)$ is a submersion, so
\[ T_{s(Q)} p''(O) = p''_* T_{(s(Q),\Lambda,Q)} O \subset p''_* T_{(\Pi(Q),\Lambda,Q)} \Phi'''. \]
Thus the maps that we claimed were transverse are so.

\section{Appendix: Calculation of \texorpdfstring{$\Hom_\Mt(\II_k,\II_{k+j})$}{Hom(I\_k,I\_{k+j})}}
In this appendix we will show that if $k \ge 2n-1$ and $j \ge 0$, the natural map $\Gamma(\AA_j|_L) \to \Hom_\Mt(\II_k, \II_{k+j})$ is an isomorphism and $\Ext^i_\MM(\II_k, \II_{k+j}) = 0$ for $i = 1, 2$, and for $i = 3$ as well if $j \ge 2n-7$.

Let $\pi: \Mt \to L$, and for brevity let $\HH = \HH om(\II_{2n}, \II_{4n})$.  There is a natural map $\pi^*(\AA_{2n}|_L) \to \HH$.  We will show that $\pi_* \HH$ is a vector bundle, that $R^1 \pi_* \HH = 0$, and that the adjunct map $\AA_{2n}|_L \to \pi_* \HH$ is an isomorphism.  Over a smooth $Q \in L$ this is clear, since if $W$ and $W'$ are maximal isotropic subspaces of opposite families, it is well-known that $\Cl_\ev = \End_\CC(I_\ev) \oplus \End_\CC(I'_\ev)$, but to show it in general will take some work.

Let $U \subset L$ be an open set over which $\Pi$ is defined, as in \S\ref{four_quadrics}.

First we show that $\pi_* \HH$ is a vector bundle and that there is no higher pushforward.  Let $Q \in L$ be corank 2 and $\ell = \pi^{-1} Q$, which is naturally identified with the cone line of $Q$.  Let $J$ be the ideal of the $\PP^n$ spanned by $\Pi(Q)$ and the cone line.  Then $\OO_\ell \otimes J_\ev \subset \II_{2n}|_\ell$.  The map
\[ \OO_\ell(-1) \otimes \II_{2n}|_\ell \to \OO_\ell \otimes J_\odd \]
given by $v \otimes \xi \mapsto v\xi$ is surjective, and the kernel is $\OO_\ell(-1) \otimes J_\ev$.  Since there are no extensions of $\OO_\ell$ by $\OO_\ell(-1)$, we have
\[ \II_{2n}|_\ell = \OO_\ell \otimes J_\ev\ \oplus\ \OO_\ell(1) \otimes J_\odd. \]
Since $\II_{4n}|_\ell = \II_{2n}|_\ell$ and $\dim J_\ev = \dim J_\odd = 2^{n-2}$, we have $\HH|_\ell = (\OO_\ell(-1) \oplus \OO_\ell^2 \oplus \OO_\ell(1))^{2^{2n-4}}$.  Since $\dim H^0(\HH|_\ell) = 2^{2n-2} = \rank \HH$ and $H^1(\HH|_\ell) = 0$, the pushforward of $\HH$ from $\Mt$ to $\MM$ is a vector bundle and there is no higher pushforward.  Since $\MM \to L$ is flat and finite, the conclusion follows.

Let $\EE \subset \OO_U \otimes V$ be the vector bundle on $U$ whose fiber over $Q \in U$ is $\Pi(Q)$, and $\EE^\perp \subset \OO_U \otimes V$ the fiberwise orthogonal.  Because $\Pi(Q)$ avoids $Q_\sing$, $\EE^\perp$ is indeed a vector bundle.  Because $\Pi(Q)$ is isotropic, $\EE \subset \EE^\perp$.  Let $P$ be the $\PP^1$-bundle $\PP(\EE/\EE^\perp)$ and $\varpi: P \to U$ the natural map.  Then $P$ carries a natural quadratic form with values in $\varpi^* \OO_U(1)$, whose zero set is naturally isomorphic to $\pi^{-1}U \subset \Mt$.  From now on let us identify them, and not distinguish between $\pi$ and $\varpi$.

Let $\JJ \subset \AA|_U$ be the ideal generated by $\Wedge^{n-1} \EE \subset \AA_{n-1}|_U$.  Then $\II_{2n} \subset \pi^* \JJ_{2n}$.  Dualizing, $\pi^* \JJ_{2n}^* \twoheadrightarrow \II_{2n}^*$.  In this paragraph we will show that the adjunct map $\JJ_{2n}^* \to \pi_* \II_{2n}$ is an isomorphism.   Let $\OO_P(-1)$ be the relative tautological line bundle and consider the sequence maps on $P$
\[ \dotsb \to \OO_P(-1) \otimes \pi^*\JJ_{2n-1} \to \OO_P \otimes \pi^*\JJ_{2n} \to \OO_P(1) \otimes \pi^*\JJ_{2n+1} \to \dotsb \]
determined by $v \otimes \xi \mapsto v\xi$.  The composition of any two, for example
\begin{align*}
\OO_P \otimes \pi^*\JJ_{2n} \to&\ \OO_P(2) \otimes \pi^*\JJ_{2n+2} \\
&= \OO_P(2) \otimes \pi^*(\JJ_{2n} \otimes \OO_U(1)),
\end{align*}
is exactly the natural quadratic form mentioned above.  Thus its restriction to the open set of $\Mt$ is exact, and the image in $\pi^* \JJ_{2n}$ is exactly $\II_{2n}$.  After dualizing the sequence above, we can get an exact sequence
\[ 0 \to \OO_P(-1) \otimes \pi^*\JJ_{2n+1}^* \to \OO_P \otimes \pi^*\JJ_{2n}^* \to \II_{2n}^* \to 0. \]
Pushing this down to $U$, we find that indeed $\JJ_{2n}^* = \pi_* \II_{2n}^*$.

On $\pi^{-1}U \subset \Mt$ we have a commutative diagram
\begin{diagram}
\pi^*(\AA_{2n}|_U) & \rTo & \II_{2n}^* \otimes \II_{4n} \\
\dTo & & \dTo \\
\pi^*(\JJ_{2n}^* \otimes \JJ_{4n}) & \rTo & \II_{2n}^* \otimes \pi^*\JJ_{4n}.
\end{diagram}
On $U$, we have the adjunct
\begin{diagram}
\AA_{2n}|_U & \rTo & \pi_*(\II_{2n}^* \otimes \II_{4n}) \\
\dTo & & \dTo \\
\JJ_{2n}^* \otimes \JJ_{4n} & \rTo & \pi_*\II_{2n}^* \otimes \JJ_{4n}.
\end{diagram}
The right vertical map is injective on each fiber, and we have just seen that the bottom map is an isomorphism, so the top map is injective on each fiber.  Since $\AA_{2n}|_U$ and $\pi_*(\II_{2n}^* \otimes \II_{4n})$ have the same rank, it is an isomorphism, as claimed.

To finish, suppose that $k$ and $j$ are both even.  Then since $k \ge 2n$ and $\II_k = \II_{2n} \otimes \pi^*\OO_L(k/2 - n)$, we have \[ \HH om(\II_k, \II_{k+j}) = \HH om(\II_{2n},\II_{4n}) \otimes \pi^*\OO_L(j/2 - n). \]
Now $\pi_*$ of this is $\AA_{2n}(j/2 - n)|_L$, which has the same global sections as $\AA_j|_L$, no middle cohomology, and no top cohomology if $j/2-n \ge -3$.  The cases where $k$ is odd or $j$ is odd are entirely similar; the bound $j \ge 2n-7$ comes from the latter.

\bibliography{quadrics}

\begin{thebibliography}{10}

\bibitem{addington_spinor}
N.~Addington.
\newblock Spinor sheaves on singular quadrics.
\newblock Preprint, arXiv:0904.1766.

\bibitem{amg}
P.~S. Aspinwall and D.~R. Morrison.
\newblock Stable singularities in string theory, with an appendix by {M}ark
  {G}ross.
\newblock {\em Comm. Math. Phys.}, 178:115--134, 1996.

\bibitem{beilinson}
A.~A. Be{\u\i}linson.
\newblock Coherent sheaves on {$\PP^n$} and problems in linear algebra.
\newblock {\em Funct. Anal. Appl.}, 12(3):214--216, 1978.

\bibitem{bgg}
I.~N. Bern\u{s}te{\u\i}n, I.~M. Gel'fand, and S.~I. Gel'fand.
\newblock Algebraic bundles over {$\PP^n$} and problems of linear algebra.
\newblock {\em Funct. Anal. Appl.}, 12(3):212--214, 1978.

\bibitem{desale_net}
U.~N. Bhosle~(Desale).
\newblock Nets of quadrics and vector bundles on a double plane.
\newblock {\em Math. Z.}, 192:29--43, 1986.

\bibitem{bo_semiorth}
A.~I. Bondal and D.~O. Orlov.
\newblock Semiorthogonal decomposition for algebraic varieties.
\newblock Preprint, alg-geom/9506012.

\bibitem{bo_sheaves}
A.~I. Bondal and D.~O. Orlov.
\newblock Derived categories of coherent sheaves.
\newblock In {\em Proceedings of the International Congress of Mathematicians},
  volume~2, pages 47--56, Beijing, 2002. Higher Ed. Press.

\bibitem{bridgeland_omegaX}
T.~Bridgeland.
\newblock Equivalences of triangulated categories and {F}ourier-{M}ukai
  transforms.
\newblock {\em Bull. London Math. Soc.}, 31(1):25--34, 1999.

\bibitem{andrei_thesis}
A.~\Caldararu.
\newblock {\em Derived categories of twisted sheaves on {C}alabi-{Y}au
  manifolds}.
\newblock PhD thesis, Cornell, 2000.

\bibitem{andrei_ell_fib}
A.~\Caldararu.
\newblock Derived categories of twisted sheaves on elliptic threefolds.
\newblock {\em J. Reine Angew. Math.}, 544:161–--179, 2002.

\bibitem{cdhps}
A.~\Caldararu, J.~Distler, S.~Hellerman, T.~Pantev, and E.~Sharpe.
\newblock Non-birational twisted derived equivalences in abelian {GLSM}s.
\newblock Preprint, arXiv:0709.3855.

\bibitem{desale_ramanan}
U.~V. Desale and S.~Ramanan.
\newblock Classification of vector bundles of rank 2 on hyperelliptic curves.
\newblock {\em Invent. Math.}, 38(2):161--185, 1976.

\bibitem{huybrechts}
D.~Huybrechts.
\newblock {\em Fourier-{M}ukai transforms in algebraic geometry}.
\newblock Oxford Mathematical Monographs. Oxford Univ. Press, 2006.

\bibitem{ingalls_khalid}
C.~Ingalls and M.~Khalid.
\newblock Moduli spaces of sheaves on {K}3 surfaces of degree 8 and their
  associated {K}3 surfaces of degree 2.
\newblock Preprint, arXiv:0803.0179.

\bibitem{kapranov_square}
M.~M. Kapranov.
\newblock The derived category of coherent sheaves on the square [sic].
\newblock {\em Funct. Anal. Appl.}, 20(2):141--142, 1986.

\bibitem{kapranov_bgg}
M.~M. Kapranov.
\newblock On the derived category and {$K$}-functor of coherent sheaves on
  intersections of quadrics.
\newblock {\em Math. USSR Izv.}, 32(1):191--204, 1989.

\bibitem{kuznetsov_quadrics}
A.~Kuznetsov.
\newblock Derived categories of quadric fibrations and intersections of
  quadrics.
\newblock {\em Adv. Math.}, 218:1340–--1369, 2008.

\bibitem{mukai_elliptic}
S.~Mukai.
\newblock Duality between {$D(X)$} and {$D(\hat X)$} with its application to
  {P}icard sheaves.
\newblock {\em Nagoya Math. J.}, 81:153--175, 1981.

\bibitem{mukai_K3}
S.~Mukai.
\newblock Moduli of vector bundles of {K}3 surfaces, and symplectic manifolds.
\newblock {\em Sugaku Expositions}, 1(2):139--174, 1988.

\bibitem{reid}
M.~Reid.
\newblock {\em The complete intersection of two or more quadrics}.
\newblock PhD thesis, Cambridge, 1972.

\bibitem{tyurin}
A.~N. Tyurin.
\newblock On intersections of quadrics.
\newblock {\em Russ. Math. Surv.}, 30:51--105, 1975.

\bibitem{vafa_witten}
C.~Vafa and E.~Witten.
\newblock On orbifolds with discrete torsion.
\newblock {\em J. Geom. Phys.}, 15(3):189--214, 1995.

\bibitem{weil}
A.~Weil.
\newblock Footnote to a recent paper.
\newblock {\em Amer. J. of Math.}, 76(2):347--350, 1954.

\end{thebibliography}
\bibliographystyle{plain}

\end{document}